\documentclass[11pt]{amsart}
\usepackage{amssymb,amsmath,amsfonts,amscd,euscript}

\newcommand{\nc}{\newcommand}

\numberwithin{equation}{section}
\newtheorem{thm}{Theorem}[section]
\newtheorem{prop}[thm]{Proposition}
\newtheorem{lem}[thm]{Lemma}
\newtheorem{cor}[thm]{Corollary}
\theoremstyle{remark}
\newtheorem{rem}[thm]{Remark}

\newtheorem{dfn}[thm]{Definition}

\nc{\gl}{\mathfrak{gl}}
\nc{\GL}{\mathfrak{GL}}
\nc{\g}{\mathfrak{g}}
\nc{\gh}{\widehat\g}
\nc{\h}{\mathfrak{h}}
\nc{\la}{\lambda}
\nc{\C}{\mathbb C }
\nc{\Z}{\mathbb Z }
\nc{\N}{\mathbb N }
\nc{\R}{\mathbb R }
\nc{\Q}{\mathbb Q }
\nc{\al}{\alpha }

\nc{\ta}{\theta}
\nc{\ve}{\varepsilon}
\nc{\ch}{{\mathop {\rm ch}}}
\nc{\Tr}{{\mathop {\rm Tr}\,}}
\nc{\Id}{{\mathop {\rm Id}}}
\nc{\ad}{{\mathop {\rm ad}}}
\nc{\bra}{\langle}
\nc{\ket}{\rangle}
\nc{\x}{{\bf x}}
\nc{\pa}{\partial}
\nc{\ld}{\ldots}
\nc{\cd}{\cdots}
\nc{\hk}{\hookrightarrow}
\nc{\T}{\otimes}
\newcommand{\bea}{\begin{equation}}
\newcommand{\ena}{\end{equation}}
\newcommand{\be}{\begin{equation*}}
\newcommand{\en}{\end{equation*}}
\nc{\gr}{\mathrm{gr}}
\nc{\ov}{\overline}

\nc{\cO}{\mathcal O}
\nc{\msl}{\mathfrak{sl}}
\nc{\mgl}{\mathfrak{gl}}
\nc{\U}{\mathrm U}
\nc{\V}{\EuScript V}
\nc{\bH}{\EuScript H}
\nc{\Res}{\mathrm{Res\ }}

\newcommand{\fn}{{\mathfrak n}}

\newcommand{\bu}{{\bf u}}
\newcommand{\eO}{\EuScript{O}}
\newcommand{\Hom}{\mathrm{Hom}}

\begin{document}

\title[Systems of correlation functions and the Verlinde algebra]
{Systems of correlation functions, coinvariants and the Verlinde algebra}

\author{Evgeny Feigin}
\address{Evgeny Feigin:\newline
Department of Mathematics, University Higher School of Economics,
20 Myasnitskaya st, 101000, Moscow, Russia
{\it and }\newline
Tamm Theory Division,
Lebedev Physics Institute,\newline
Leninisky prospect, 53,
119991, Moscow, Russia}
\email{evgfeig@gmail.com}

\begin{abstract}
We study the Gaberdiel-Goddard spaces of systems of correlation functions
attached to an affine Kac-Moody Lie algebra $\gh$. We prove
that these spaces are isomorphic to the spaces
of coinvariants with respect to certain subalgebras of $\gh$. This allows to
describe the Gaberdiel-Goddard spaces as  direct sums of tensor
products of irreducible $\g$-modules with multiplicities given by fusion coefficients. We thus
reprove and generalize Frenkel-Zhu's theorem.
\end{abstract}

\maketitle

\section*{Introduction}
Let $\V$ be a vertex operator algebra. In \cite{Z} Zhu introduced an associative  algebra
$A(\V)$ which captures important information about representation theory
of $\V$ (see \cite{FZ}, \cite{DM}, \cite{GG}).
The algebra $A(\V)$ is defined as a quotient of the vacuum module of $\V$. More generally,
Zhu constructed a set of $A(\V)$-modules $A(\bH)$ labeled by $\V$-modules $\bH$.
The spaces $A(\bH)$ are defined as quotients of $\bH$. It was shown in \cite{FZ} that
for the VOAs associated with integrable irreducible representations of affine Kac-Moody
algebras, these quotients can be described via fusion coefficients (structure constants of the
Verlinde algebra).

It is known that the theory of VOAs is a key ingredient in mathematical description
of the models of conformal field theory (see \cite{G}, \cite{DMS}).
In \cite{GabGod} Gaberdiel and Goddard suggested
an axiomatic approach to CFT based on the systems of correlation functions. These systems
depend on a parameter $\bu=(u_1,\dots,u_n)\in(\mathbb{CP}^1\setminus\{0\})^n$. It was shown in
\cite{GabGod}, \cite{N}, \cite{GN} that the systems in question are in one-to-one
correspondence with quotients $A_\bu(\bH)=\bH/O_\bu(\bH)$ of representations $\bH$ of a given CFT
(representations of the attached VOA) by certain subspaces $O_\bu(\bH)$.
An important fact is that Zhu's modules $A(\bH)$ coincide with $A_{(-1,\infty)}(\bH)$.

The goal of this paper is to study the Gaberdiel-Goddard spaces $A_\bu(\bH)$ for the WZW models
on non-negative integer level (or, equivalently, for $\bH$ being integrable irreducible
representation of an affine Kac-Moody Lie algebra $\gh$). Recall that $\gh$ is the
non-trivial central extension of the Lie algebra $\g\T\C[t,t^{-1}]$. Our main theorem is the description
of the spaces  $O_\bu(\bH)$:

\begin{thm}\label{H}
Let $\bH$ be an integrable irreducible $\gh$-module. Then
\[
O_\bu(\bH)=\Bigl( \g\T\prod_{j=1}^n (t^{-1}-u_j^{-1})\C[t^{-1}]\Bigr)\bH.
\]
\end{thm}

Therefore the spaces $A_\bu(\bH)$ are isomorphic to the coinvariants with respect to the subalgebra
$\g\T\prod_{j=1}^n (t^{-1}-u_j^{-1})\C[t^{-1}]$.
Using the results from \cite{FKLMM} and \cite{Fi} we obtain the following corollary:
\begin{cor}\label{cor}
Let $u_i\ne u_j$ for $i\ne j$. Then
\begin{equation}\label{N}
A_\bu(L_\la)\simeq
\bigoplus_{\mu_1,\dots,\mu_n\in P_+^k} N_{\mu_1,\dots,\mu_n}^{\la;k} V_{\mu_1}\T \dots \T V_{\mu_n}.
\end{equation}
\end{cor}
Here $\bH=L_\la$ is an integrable irreducible highest weight $\gh$-module, $V_{\mu_i}$ are irreducible
finite-dimensional $\g$-modules of highest weights $\mu_i$ and $N_{\mu_1,\dots,\mu_n}^{\la;k}$
are structure constants of the level $k$ Verlinde algebra attached to $\gh$.
Theorem \ref{H} together with  Corollary \ref{cor}
reprove and generalize  the Frenkel-Zhu theorem \cite{FZ} (treating the $n=2$ case) to an arbitrary
number of pairwise distinct points.

One of the most intriguing questions related to Zhu's and Gaberdiel-Goddard's spaces is as
follows (see \cite{GG}, \cite{FFL}, \cite{FL}): what happens when the points $u_j$ are allowed to coincide?
Using Theorem \ref{H} above and the results of \cite{FKLMM} we show that for $\g=\msl_2$ the isomorphism
\eqref{N} holds even if some points do coincide.

Our paper is organized as follows:\\
In Section $1$ we recall the generalities about
affine Kac-Moody algebras and vertex operator algebras.\\
In Section $2$ we prove Theorem \ref{H}.\\
In Section $3$ we discuss the isomorphism between the
coinvariants and the right hand side of \eqref{N}.\\
Finally, in Section $4$, we derive corollaries from Theorem \ref{H} and discuss generalizations.

\section{Affine Kac-Moody algebras and vertex operator algebras}
In this section we collect definitions and properties of affine Kac-Moody algebras
and vertex operator algebras (VOA for short) we use in the main body of the paper.
Our references here are \cite{K1}, \cite{BF}, \cite{K2}.

\subsection{Affine Kac-Moody algebras}
Let $\g$ be a simple Lie algebra with the Cartan decomposition $\g=\fn\oplus\h\oplus \fn^-$.
Let $\omega_1,\dots,\omega_l$ and $\al_1,\dots,\al_l$ be the fundamental weights and simple roots
of $\g$, $\omega_i,\al_i\in\h^*$. The weights $\omega_i$ generate the weight lattice $P$.
Its dominant cone $\{\sum_{i=1}^l s_i\omega_i, \ s_i\in\Z_{\ge 0}\}$ is denoted by $P_+$.
We denote by $\triangle_+$ the set of positive roots of $\g$.
Let  $\theta\in \triangle_+$ be the highest
root (the highest weight of the adjoint representation). For $k\in\Z_{\ge 0}$ we set
\[
P_+^k=\{\la\in P_+:\ (\la,\theta)\le k\},
\]
where $(\cdot,\cdot)$ is the Killing form normalized by $(\theta,\theta)=2$. For $\la\in P_+$ let
$V_\la$ be the irreducible $\g$-module with highest weight $\la$ and highest weight vector $v_\la\in V_\la$.
In particular, $\fn v_\al=0$ and $V_\la=\U(\fn^-)v_\la$, where $\U(\fn^-)$ is the universal enveloping
algebra. For a weight $\la\in P_+$ we denote by $\la^*\in P_+$ the highest weight of the dual
module: $V_{\la^*}\simeq V_\la^*$.

Let $\gh=\g\T\C[t,t^{-1}]\oplus \C K$ be the affine Kac-Moody algebra associated with $\g$.
The bracket on $\gh$ is given by
\[
[x\T t^i, y\T t^j]=[x,y]\T t^{i+j} + i\delta_{i+j,0}(x,y)K, \ x,y\in\g, i,j\in\Z
\]
and $K$ is central. Let $\gh=\widehat{\fn}\oplus \widehat{\h}\oplus \widehat{\fn}^-$ be the Cartan
decomposition. We have
$$
\widehat{\fn}= \fn\T 1 \oplus \g\T t\C[t],\
\widehat{\h}=\h\T 1\oplus \C K,\
\widehat{\fn}^-= \fn^-\T 1 \oplus \g\T t^{-1}\C[t^{-1}].
$$
Since $K$ is central, it acts as a scalar
on any irreducible $\gh$-module. This scalar is called the level of a module.
For each $\la\in P_+^k$ there exists the highest weight irreducible $\gh$-module $L_\la$
of level $k$ with a highest weight vector $l_\la\in L_\la$ such that
$\widehat{\fn}l_\la=0$ and $\U(\widehat{\fn}^-)l_\la=L_\la$.

The modules $L_\la$, $\la\in P_+^k$ form a complete list of integrable irreducible
highest weight representations of level $k$. We note that $L_\la$ carries a natural degree (or energy)
grading $L_\la=\bigoplus_{d\ge 0} L_\la(d)$ defined as follows:
$L_\la(0)$ contains $l_\la$ and $x\T t^i$ acts from $L_\la(d)$ to $L_\la(d-i)$
for all $x\in\g$ and $i\in\Z$. In particular, $L_\la(0)=V_\la$.

\subsection{Vertex operator algebras $\V_k(\g)$}
Recall that a vertex operator algebra is a collection of fields
$Y(A,z)$ subject to certain conditions. The fields are labeled by vectors $A\in V$,
where $V$ is some infinite-dimensional vector space. The space $V$ carries a grading
$V=\bigoplus_{d\ge 0} V(d)$. For an element $A\in V(d)$ we set $|A|=d$. Each field $Y(A,z)$ is
a series in $z,z^{-1}$ with coefficients in $\mathrm{End}(V)$:
\[
Y(A,z)=\sum_{m\in\Z} A_{(m)} z^{-m-|A|}, \ A_{(m)}\in \mathrm{End} V.
\]
There is a special vacuum vector $|0\rangle\in V$ which spans $V(0)$. The corresponding field
is trivial: $Y(|0\rangle,z)=\mathrm{Id}$. Two most important properties of the fields are as follows:
\begin{enumerate}
\item $\forall A,B\in V$ $\exists N\in\Z$ such that $\forall m>N$ $A_{(m)}B=0$; \label{field}
\item $\forall A,B\in V$ $\exists N\in\Z_{>0}$ such that $(z-w)^N [Y(A,z),Y(B,w)]=0$.
\end{enumerate}
The second property is referred to as the locality property.
In addition, the space of fields is endowed with the structure of the operator product expansion (OPE).
More precisely, for any two elements $A,B\in V$ the following equality holds
\[
Y(A,z)Y(B,w)=\sum_{n\in\Z} Y(A_{(n)}B,w)(z-w)^{-n-|A|}.
\]
The OPE is used to define the notion of the representation of a vertex operator algebra: 
a vector space $M$ endowed with a set of $\mathrm{End}(M)$-valued series 
$Y_M(A,z)$, $A\in V$  is called a $V$-modules if the following relations hold:
\[
Y_M(A,z)Y_M(B,w)=\sum_{n\in\Z} Y_M(A_{(n)}B,w)(z-w)^{-n-|A|}.
\]
For example, the vertex operator algebra $V$ itself together with the series $Y(A,z)$
form the so-called adjoint representation.

It turns out that the level $k$ vacuum $\gh$-module $L_0$ can be endowed with the structure of a VOA.
This VOA is denoted by $\V_k(\g)$. The grading is given by the energy grading and the vacuum vector is $l_0$.
The precise formulas for the fields are as follows.
First, for $a\in\g$
\[
Y(a\T t^{-1}l_0,z)=\sum_{m\in \Z} (a\T t^m)z^{-m-1},
\]
i.e. $(a\T t^{-1}l_0)_{(m)}=a\T t^m$.
In general, setting $a(z)=Y(a\T t^{-1}l_\la,z)$, one gets for $a_i\in\g$, $s_i>0$:
\begin{multline}\label{VOA}
Y(a_1\T t^{-s_1}\cdots a_N\T t^{-s_N}l_0,z)\\ =\frac{1}{(s_1-1)!\dots (s_N-1)!}
: \pa^{s_1-1} a_1(z)\dots \pa^{s_N-1} a_N(z):,
\end{multline}
where $\pa$ is the $z$-derivative and $:\ :$ denotes the normally ordered product.
For two fields $Y(A,z)$ and $Y(B,z)$ their normally ordered product is defined by the formula
\begin{multline}\label{::}
:Y(A,z)Y(B,z):\\ =\sum_{m_2\in\Z} \left(
\sum_{m_1\le -|A|} A_{(m_1)}B_{(m_2)} z^{-m_1-|A|} +
\sum_{m_1> -|A|} B_{(m_2)}A_{(m_1)} z^{-m_1-|A|}
\right) z^{-m_2-|B|}.
\end{multline}

\begin{rem}
Thanks to the property \eqref{field}, the normally ordered product of two fields is well defined.
\end{rem}

The normally ordered product of $N$ fields is defined recursively:
\[
:Y(A,z)Y(B,z)Y(C,z):\ =\ :Y(A,z)(:Y(B,z)Y(C,z):):
\]
and so on.


\begin{rem}
Formula \eqref{VOA} is not specific for $\V_k(\g)$. This is a particular case of the general
reconstruction procedure for getting formulas for all fields starting from the generating ones.
\end{rem}

The space $L_0$ carries a structure of a representation of the Virasoro algebra 
given by the Segal-Sugawara construction. Let $T_i$, $i\in\Z$ be the generators 
of the Virasoro algebra. Then one has
\begin{equation}\label{Vir}
Y(T_{-1}A,z)=\pa Y(A,z)
\end{equation}
for all $A\in L_0$.

For a series $f(z)=\sum_{i\in\Z} f_i z^i$ its formal residue is defined as a coefficient in front of $z^{-1}$:
$\Res f(z)=f_{-1}$. Obviously, $\Res \pa f(z)=0$.

\section{Systems of correlation functions}
Let $\bu=(u_1,\dots,u_n)\in (\mathbb{CP}^1\setminus \{0\})^n$
be a collection of (non necessarily distinct) points.
In \cite{GabGod} the authors defined certain quotient $A_{\bu}$ of
representations of a given VOA.
These quotients depend on $\bu$ and describe systems of correlation function
on the sphere (see \cite{GN}). The systems in question play the crucial role in the axiomatic
approach to CFT developed in \cite{GabGod}
(see also \cite{GN}, \cite{G}). Let us recall the main definitions.

Let $\bH$ be a representation of a VOA $\V$.
Assume for a moment that $u_1=\infty$ and $u_i\ne \infty, i>1$. Then the Gaberdiel-Goddard
space $A_\bu(\bH)$ is defined as
$$A_{\bu}(\bH)=\bH/O_{\bu}(\bH),$$
where the space $O_{\bu}(\bH)=O_{(u_1,\dots,u_n)}(\bH)\hk \bH$
is spanned by the vectors.
\begin{equation}\label{Y}
\Res (z^{-1-M+(2-n)|A|}\prod_{j=2}^{n}(z-u_j)^{|A|}Y(A,z)v)
\end{equation}
for all $M>0, A\in\V, v\in \bH.$
In addition, the family of spaces $A_\bu(\bH)$ is invariant with respect to the diagonal action of the group
$PSL_2$ on $\bu\in (\mathbb{CP}^1)^n$. This property is called the M\"obius invariance.
Using the M\"obius invariance in what follows we assume without loss of generality that
$u_1=\infty$.

It is convenient for us to rewrite formula \eqref{Y} slightly to exclude the restriction
$u_i\ne \infty, i>1$.
Namely, let
\begin{equation}\label{YM}
Y_{\bu}^{(M)}(A)=\Res (z^{-1-M+|A|}\prod^n_{j=2}(z^{-1}-u^{-1}_j)^{|A|}Y(A,z)).
\end{equation}
Clearly $Y_{\bu}^{(M)}(A)v$ is proportional to \eqref{Y} if $u_i\ne \infty, i>1.$
\begin{dfn}
For each $n$-tuple $\bu=(u_1,\dots,u_n)\in (\mathbb{CP}^1\setminus\{0\})^n$, $u_1=\infty$
the subspace  $O_\bu(\bH)\hk\bH$ is a linear span of the elements
$Y_{\bu}^{(M)}(A)v$ with $A\in \V$, $M>0$, $v\in\bH$. The corresponding quotient space is denoted
by $A_\bu(\bH)$: $A_{\bu}(\bH)=\bH/O_u(\bH)$.
\end{dfn}

\begin{rem}
If  $u_i=\infty$ for some $i$, we put $u_i^{-1}=0$.
\end{rem}

In what follows we are only concerned with vertex operator algebras $\V_k(\g)$
attached to an affine Kac-Moody algebras $\gh$ on level $k\in\Z_{\ge 0}$.
For a weight $\la\in P_+^k$ we denote
$O_{\bu}(\la)=O_{\bu}(L_\la), A_{\bu}(\la)=A_{\bu}(L_\la).$
Our goal is to prove that
\begin{equation}\label{=}
O_{\bu}(\la)=(\g\T\prod^n_{j=1}(t^{-1}-u^{-1}_j)\C[t^{-1}])L_{\la}
\end{equation}
(recall that $u_1=\infty$).
We denote the right hand side by $O^c_{\bu}(\la)$ (the upper index $c$ is because of coinvariants,
see Section \ref{coinvariants}). 

\begin{rem}
If $u_j=\infty$ for all $j=1,\dots,n$, then 
\[
O_\bu^c(\la)=(\g\T t^{-n}\C[t^{-1}])L_\la,
\]
which agrees with the general formula for $O_{(\infty,\dots,\infty)}(\la)$, see
for example \cite{GN}.
\end{rem}

We start  with the following simple but important lemma:
\begin{lem}\label{Res}
For any $a\in \g, M\in \Z$ and $p_1,\dots,p_N \in\C$
$$
\Res \left( z^M \prod_{j=1}^N(z^{-1}-p_j)Y(a\T t^{-1},z)\right)=a\T t^M \prod^N_{j=1}(t^{-1}-p_j).
$$
\end{lem}
\begin{proof}
Follows from $Y(a\T t^{-1},z)=\sum_{i\in \Z}z^{-i-1}(a\T t^i)$.
\end{proof}

We first show that the left hand side of \eqref{=} contains the right hand side.
\begin{lem}
$O^c_{\bu}(\la)\hk O_\bu(\la).$\label{1}
\end{lem}
\begin{proof}
Take $a\in \g$. Since $|a\T t^{-1}|=1$ we have
$$
Y_{\bu}^{(M)}(a\T t^{-1})=\Res z^{-M}\prod^n_{j=2}(z^{-1}-u_j^{-1})Y(a\T t^{-1},z).
$$
Now our lemma follows from Lemma \ref{Res} and definition of $O_{\bu}(\la).$
\end{proof}

We now want to prove the reverse inclusion $O_\bu(\la)\hk O^c_\bu(\la)$.
We say that an operator $X\in \mathrm{End}(L_\la)$ is good, if
\[
\mathrm{Im}(X)\hk (\g\T\prod^n_{j=1}(t^{-1}-u^{-1}_j)\C[t^{-1}])L_{\la}.
\]
Of course, being good or not depends on $\bu$ and $\bH$. In what follows we fix $\bu$ and $\bH$,
so we say simply "good". Then the inclusion
\[
O_{\bu}(\la)\hk O^c_\bu(\la)
\]
is equivalent to the statement that all operators $Y_\bu^{(M)}(A,z)$,
$M>0$, $A\in\V$ are good.

We first consider the elements
$A=a\T t^{-s}, s\ge 1$.
\begin{lem}\label{a_s}
For any $M>0, s>0$ and $a\in\g$ the operators $Y^{(M)}_{\bu}(a\T t^{-s})$ are good.
\end{lem}
\begin{proof}
Since $|a\T t^{-s}|=s$ we obtain (see formula \eqref{VOA})
\begin{multline*}
(s-1)! Y^{(M)}_{\bu}(a\T t^{-s})=\Res z^{-1-M+s}
\prod^n_{j=2}(z^{-1}-u_j^{-1})^s\pa^{s-1}{Y(a\T t^{-1},z)}\\
=(-1)^{s-1}\Res \left( \pa^{s-1} \left[(z^{-1-M+s}\prod^n_{j=2}(z^{-1}-u^{-1}_j)^s)\right]Y(a\T t^{-1},z)\right).
\end{multline*}
We note that
\begin{multline*}
\pa^{s-1}\left(z^{-1-M+s}\prod^n_{j=2}(z^{-1}-u^{-1}_j)^s\right)\\=
\sum_{d_1+\dots + d_n=s-1} c_{\bf d}(\pa^{d_1} z^{-1-M+s})
\prod^n_{j=2}\pa^{d_j}{(z^{-1}-u^{-1}_j)^s}\\ =
\sum 
c_{{\bf d},{\bf b}} z^{-1-M+s-d_1}z^{ -\sum^n_{j=2} d_j-\sum^n_{j=2} b_j} \prod^n_{j=2}(z^{-1}-u^{-1}_j)^{s-b_j},
\end{multline*}
where $c_{\bf d}, c_{\bf d,\bf b}$ are some constants and the sum runs over
$${\bf d}=(d_1,\dots,d_n)\in\Z_{\ge 0}^n, \ {\bf b}=(b_2,\dots,b_n)\in\Z_{\ge 0}^{n-1}$$
such that $\sum_{i=1}^n d_i=s-1$ and
$b_i\le d_i$.
Therefore the operator $Y^{(M)}_{\bu}(a\T t^{-s})$ is equal to the
linear combination of terms
$$\Res \left(z^{-M-\sum^n_{j=2}b_j} \prod^n_{j=2}(z^{-1}-u^{-1}_j)^{s-b_j} Y(a\T t^{-1},z)\right).$$
Since $0\le b_j \le s-1$, our lemma follows from Lemma \ref{Res}.
\end{proof}

In general, we want to prove that for any
$a_1,\dots,a_m \in\g, s_1,\cdots,s_m>0$ and $M>0$ we have
\begin{equation}\label{gen}
Y^{(M)}_{\bu}(a_1\T t^{-s_1}\dots a_m\T t^{-s_m})L_\la\hk O^c_{\bu}(\la).
\end{equation}
We first reduce \eqref{gen} to the case to $s_1=1$.
\begin{lem}\label{s=1}
Assume \eqref{gen} holds for $s_1=1$ and arbitrary $a_1,\dots,a_m\in\g$ and $s_2,\dots,s_m>0$.
Then \eqref{gen} holds for arbitrary $s_1>0$ as well.
\end{lem}
\begin{proof}
Let
$B(z)=Y(a_2\T t^{-s_2}\cdots a_m\T t^{-s_m}l_\la,z)$
and $s=s_2+\dots +s_m$. Then
\begin{multline*}
(s_1-1)!Y^{(M)}_{\bu}(a_1\T t^{-s_1}\dots a_m\T t^{-s_m})\\
=\Res z^{-1-M+s_1+s}\prod^n_{j=2}(z^{-1}-u^{-1}_j)^{s_1+s}:(\pa^{s_1-1} a_1(z))B(z):,
\end{multline*}
where $a_1(z)=\sum_{i\in \Z}(a_1\T t^i)z^{-i-1}$.
Since the residue of a full derivative vanishes, we obtain for some constants $c_{l_1,l_2,l_3}$
\begin{multline*}
Y^{(M)}_{\bu}(a_1\T t^{-s_1}\cdots a_m\T t^{-s_m})\\
=\Res  \sum_{\substack{l_1+l_2+l_3=s_1 -1 \\ l_1,l_2,l_3\ge 0}}c_{l_1,l_2,l_3}
(\pa^{l_1} z^{-1-M+s_1+s})\\
\times (\pa^{l_2}\prod^n_{j=2}(z^{-1}-u_j^{-1})^{s_1+s}) :a_1(z)\pa^{l_3}_zB(z):.
\end{multline*}
We note (see \eqref{Vir}) that
$\pa^{l_3} B(z)$ is a linear combination of fields $Y(b_p,z)$ of weight
$l_3+s$, i.e.
$\pa^{l_3} B(z)=\sum^{m-1}_{p=1}\alpha_p Y(b_p,z)$, where
$\alpha_p$ are some constants and each $b_p$ is of the form
\[
b_p=a_2\T t^{-\bar s_2}\cdots a_m\T t^{-\bar s_m}, \quad
\bar s_2 + \cdots + \bar s_m=s+l_3,\ \bar s_i>0.
\]
Therefore, for some constants $\bar c_{l_1,l_2,l_3}$ we have
\begin{multline*}
Y^{(M)}_{\bu}(z) =
\Res \sum_{\substack{l_1+l_2+l_3=s-1\\p=1,\dots,m-1}}\bar c_{l_1,l_2,l_3}
z^{(-1-M-l_1+s_1-1-l_3)+1+s+l_3}\\ \times\pa^{l_2}\prod^n_{j=2}(z^{-1}-u^{-1}_j)^{s_1+s}
:a_1(z) Y(b_p,z):.
\end{multline*}
We note that since $l_2\le s_1-1$, the derivative $\pa^{l_2}\prod^n_{j=2}(z^{-1}-u^{-1}_j)^{s_1+s}$
is a linear combination of terms
$z^{-k}\prod^n_{j=2}(z^{-1}-u^{-1}_j)^{s_1-k_j+s}$ for some constants satisfying
$0\le k_2,\dots,k_n\le l_2$ and $k>l_2$.
We thus obtain that $Y^{(M)}_{\bu}(z)$ is equal to a linear combination of terms
of the form $(M_1>0)$
$$\Res (z^{-1-M_1+1+s+l_3}\prod^n_{j=2}(z^{-1}-u^{-1}_j)^{1+s+l_3}:a_1(z)Y(b_p,z):).$$
This residue is equal to
$$Y_\bu^{(M_1)}(a_1\T t^{-1} \bar a_2\T t^{-\bar s_2}\dots \bar a_m\T t^{-\bar s_m}),$$
for some $\bar a_i$, $\bar s_i$.
Thus inclusion \eqref{gen} is reduced to the case $s_1=1$.
\end{proof}

In what follows we need an explicit form of operators $Y^{(M)}_\bu(B)$. Let us introduce constants
$N^{(s)}$ and $\al_p^{(s)}$, $p=0,\dots,N^{(s)}$ by the formula:
\[
N^{(s)}=(n-1)s,\quad \sum^{N^{(s)}}_{p=0}\alpha^{(s)}_pz^{-p}=\prod^n_{j=2}(z^{-1}-u^{-1}_j)^s.
\]
For example, $\al_0^{(s)}=(-1)^{(n-1)s}\prod_{j=2}^n u_j^{-s}$ and $\al_{N^{(s)}}^{(s)}=1$.
\begin{lem}\label{YB}
 Let $|B|=s$. Then
\[
Y^{(M)}_\bu(B)=\sum_{p=0}^{N^{(s)}} \al_p^{(s)} B_{(-M-p)}.
\]
\end{lem}
\begin{proof}
Straightforward.
\end{proof}

We now prove the following proposition.
\begin{prop}\label{induction}
Let $B(z)=Y(B,z)$ be a  field such that $Y^{(M)}_{\bu}(B)$ is good.
Then for any $a\in\g$ and $M>0$ the operator  $Y^M_{\bu}(a\T t^{-1}B)$ is also good.
\end{prop}
\begin{proof}
By definition (see \eqref{VOA}) we have
$$Y(a\T t^{-1}Bl_\la,z)=:a(z)B(z):=\sum_{i<0}(a\T t^i)z^{-i-1}B(z)+\sum_{i\ge 0}B(z)(a\T t^i)z^{-i-1}.$$
Let $|B|=s$.
Then
\begin{multline}\label{tt}
Y^{(M)}_{\bu}(a\T t^{-1}B)=
\Res z^{-1-M+s+1} \prod^n_{j=2}(z^{-1}-u^{-1}_j)^{s+1}\sum_{i<0}(a\T t^i)z^{-i-1}B(z)\\
+\Res z^{-1-M+s+1}\prod^n_{j=2}(z^{-1}-u^{-1}_j)^{s+1}\sum_{i\ge 0}B(z)(a\T t^i)z^{-i-1}.
\end{multline}
We first consider the second summand, which is equal to
$$\sum_{i\ge 0}\Res z^{-1-M-i+s}\prod^n_{j=2}(z^{-1}-u^{-1}_j)^{s+1}B(z)(a\T t^i).$$
Since $\prod^n_{j=2}(z^{-1}-u^{-1}_j)$ is a polynomial in $z^{-1}$
and $Y^{(M)}_\bu(B)$ is good for any $M>0$, we get for any $i\ge 0$
$$
\Res \left( z^{-1-M-i+s}\prod^n_{j=2}(z^{-1}-u^{-1}_j)^{s+1}B(z)\right) L_\la\hk O^c_{\bu}(\la).
$$
Therefore the second summand of the right hand side of \eqref{tt} is good.

Now consider the first summand of the right hand side of \eqref{tt}.
Using the decomposition $$B(z)=\sum_{m\in \Z}B_{(m)}z^{-m-s},$$
it can be rewritten as
$$\sum_{m\in \Z} \Res z^{-M-m}\prod^n_{j=2}(z^{-1}-u^{-1}_j)^{s+1}\sum_{i<0}(a\T t^i)z^{-i-1}B_{(m)}.$$
Introducing a new variable $l$ such that $l=-M-m$, we rewrite further
\begin{equation}\label{aB}
\sum_{l\in \Z} \Res (z^l\prod^n_{j=2}(z^{-1}-u^{-1}_j)^{s+1}\sum_{i<0}(a\T t^i)z^{-i-1}B_{(-M-l)}).
\end{equation}

We consider three cases:
$$l<0,\qquad l\ge(n-1)(s+1), \qquad 0\le l\le (n-1)(s+1)-1.$$
The first two cases are simple and the last one is more involved.

Assume first that $l<0$.
Then $z^l\prod^n_{j=2}(z^{-1}-u^{-1}_j)^{s+1}$ is a polynomial in $z^{-1}$ without constant term and therefore the corresponding term in \eqref{aB} is equal to
$$\Res z^l\prod^n_{j=2}(z^{-1}-u^{-1}_j)^{s+1}a(z)B_{(-M-l)}.$$
Since $\prod^n_{j=2}(z^{-1}-u^{-1}_j)^{s+1}$ is a polynomial in $z^{-1},$ we conclude from Lemma \ref{a_s}
that all terms of \eqref{aB} with $l< 0$ are good.

Assume that $l\ge(n-1)(s+1).$ Then $z^l\prod^n_{j=2}(z^{-1}-u^{-1}_j)^{s+1}$  is a polynomial in $z$ and hence
$$\Res \left(z^l\prod^n_{j=2}(z^{-1}-u^{-1}_j)^{s+1}\sum_{i<0}(a\T t^i)z^{-i-1}\right)$$
vanishes.

We are left to show that the sum of the terms of \eqref{aB} with   $0\le l\le (n-1)(s+1)-1$
is good. This sum is equal to
\begin{multline*}
\sum_{l=0}^{N^{(s+1)}-1} \Res z^l\prod^n_{j=2}(z^{-1}-u^{-1}_j)^{s+1}\sum_{i<0}(a\T t^i)z^{-i-1}B_{(-M-l)}\\
=\Res\sum_{l=0}^{N^{(s+1)}-1}\sum_{p=0}^{N^{(s+1)}}\sum_{i<0} z^{l-p-i-1} \al_p^{(s+1)} (a\T t^i) B_{(-M-l)}\\
= \sum_{l=0}^{N^{(s+1)}-1}\sum_{i=1}^{N^{(s+1)}-l} \al_{i+l}^{(s+1)} (a\T t^{-i}) B_{(-M-l)}.
\end{multline*}
Therefore our goal is to show that the operator
\begin{equation}\label{s+1}
\sum^{N^{(s+1)}-1}_{l=0}\sum^{N^{(s+1)}-l}_{i=1}\alpha^{(s+1)}_{i+l}a_{(-i)}B_{(-M-l)}
\end{equation}
is good, where $a_{(-i)}=a\T t^{-i}$.
Recall (see Lemma \ref{YB}) that
\[
Y_\bu^{(M)}(a\T t^{-1})=\sum_{p=0}^{N^{(1)}} a_{(-M-p)}\al_p^{(1)}\ \text{ and }\
Y_\bu^{(M)}(B)=\sum_{p=0}^{N^{(s)}} B_{(-M-p)}\al_p^{(s)}.
\]
We know that for all $M>0$ these operators are good.
Note that if $X\in \mathrm{End} (L_\la)$
is good then we have
\begin{itemize}
\item for any $A\in \mathrm{End} (L_\la)$ the operator $XA$ is good;
\item if $i\le 0$ then for any $a\in\g$ the operator $(a\T t^i)X$ is good.
\end{itemize}
Therefore, in order to prove that \eqref{s+1} is good it suffices to represent \eqref{s+1}
as a linear combination of operators of the form $Y_\bu^{(M)}(a\T t^{-1})B_{(k_M)}$ and
$a_{(-i_M)}Y_\bu^{(M)}(B)$,
where $M>0$, $i_M, k_M\ge 0$. The existence of such presentation follows from
the following statement: there exist two polynomials $p(x,y)$ and $q(x,y)$ in two variables such that
\begin{equation}\label{pq}
\sum_{l=0}^{N^{(s+1)}-1}\sum_{i=1}^{N^{(s+1)}-l} \al_{i+l}^{(s+1)}x^iy^l=
\prod_{j=2}^n (x-u_j^{-1})p(x,y) + \prod_{j=2}^n (y-u_j^{-1})^s q(x,y).
\end{equation}
In fact, assume that such $p(x,y)$ and $q(x,y)$ exist, say
\[
p(x,y)=\sum_{m,r\ge 0} p_{m,r}x^my^r,\quad
q(x,y)=\sum_{m,r\ge 0} q_{m,r}x^my^r.
\]
Then it is easy to see that
\begin{multline*}
\sum^{N^{(s+1)}-1}_{l=0}\sum^{N^{(s+1)}-l}_{i=1}\alpha^{(s+1)}_{i+l}a_{(-i)}B_{(-M-l)}\\=
\sum_{m,r\ge 0} p_{m,r} Y^{(m+1)}_\bu(a\T t^{-1}) B_{(-r-M)} +
\sum_{m,r\ge 0} q_{m,r} a_{(-m-1)}Y^{(M+r)}_\bu(B).
\end{multline*}
Since the right hand side is good, the existence of $p,q$ as in \eqref{pq} implies
our proposition. We prove the existence in a separate lemma.
\end{proof}

\begin{lem}
Polynomials $p(x,y)$ and $q(x,y)$ satisfying \eqref{pq} exist.
\end{lem}
\begin{proof}
Let us slightly modify the left hand side of \eqref{pq} by adding the terms
$\sum_{l=0}^{N^{(s+1)}} \al_l y^l$ (which correspond to the value $i=0$ in \eqref{pq}).
The sum we add is equal to $\prod_{j=2}^n (y-u_j^{-1})^{s+1}$ and therefore it suffices to prove
our lemma replacing the left hand side of \eqref{pq} by
$\sum_{l,i\ge 0} \al_{i+l}^{(s+1)} x^i y^l$, where $\al_r^{(s+1)}=0$ if $r>N^{(s+1)}$.

For a polynomial $f(v)=\sum_{r\ge 0} \beta_r v^r$ we define a polynomial $\varphi(f)$
in two variables $x,y$ by the formula
\[
\varphi(f)=\sum_{i,l\ge 0} \beta_{i+l} x^i y^l.
\]
It is easy to show that for any $u\in\C$
\[
\varphi(f(v)(v-u))=\varphi(f)(x-u) + yf(y).
\]
Using this equality, we compute
\begin{align*}
\varphi\left(\prod_{j=2}^n (v-u_j^{-1})^{s+1}\right)=&
(x-u_n^{-1})\varphi\left(\frac{\prod_{j=2}^n (v-u_j^{-1})^{s+1}}{(v-u_n^{-1})}\right) \\
&+ \frac{y\prod_{j=2}^n (y-u_j^{-1})^{s+1}}{y-u_n^{-1}}.
\end{align*}
The second term is already of the form of the right hand side of \eqref{pq}. The first term
is equal to
\[
(x-u_n^{-1})\left[
(x-u_{n-1}^{-1})\varphi\left(\frac{\prod_{j=2}^n (v-u_j^{-1})^{s+1}}{(v-u_n^{-1})(v-u_{n-1}^{-1})}\right)
+ \frac{y\prod_{j=2}^n (y-u_j^{-1})^{s+1}}{(y-u_n^{-1})(y-u_{n-1}^{-1})}
\right].
\]
Again, the second term is already of the needed form. After $n$ iterations of this
procedure (omitting terms as above) we obtain as a result
\[
\prod_{j=2}^n (x-u_j^{-1})\varphi\left(\prod_{j=2}^n (v-u_j^{-1})^s\right) + y\prod_{j=2}^n (y-u_j^{-1})^s.
\]
This proves the lemma.
\end{proof}

We are now ready to prove the main theorem.
\begin{thm}\label{main}
$O_\bu(\la)=O^c_\bu(\la).$
\end{thm}
\begin{proof}
Let us show that $O_\bu(\la)\hk O^c_\bu(\la).$  This is equivalent to the statement that
all operators
\begin{equation} \label{good}
Y_\bu^{(M)} (a_1\T t^{-i_1}\dots a_m\T t^{-i_m}),\ M>0,\ a_j\in\g,\ i_1,\dots,i_m>0
\end{equation}
are good. We prove this statement by induction on $m$. For $m=1$ it follows from Lemma \ref{a_s}.
The induction step can be performed using Proposition \ref{induction} and Lemma \ref{s=1}.

The reverse inclusion $O^c_\bu(\la)\hk O_\bu(\la)$ is proved in Lemma \ref{1}.
\end{proof}

\section{Coinvariants}\label{coinvariants}
In this section we discuss the following theorem:
\begin{thm}\label{coinv}
Let $L_\la$ be an integrable irreducible $\gh$-module with highest weight $\la\in P_+^k$.
Then we have an isomorphism of $\g^{\oplus n}$-modules
\begin{equation}\label{Vn}
L_\la/\g\T\prod_{j=1}^n (t^{-1}-u_j^{-1})\C[t^{-1}]\simeq
\bigoplus_{\mu_1,\dots,\mu_n\in P_+^k} N_{\mu_1,\dots,\mu_n}^{\la;k} V_{\mu_1}\T \dots \T V_{\mu_n},
\end{equation}
where $N_{\mu_1,\dots,\mu_n}^{\la;k}$ are level $k$ Verlinde numbers and
$u_j\in\mathbb{CP}^1\setminus\{0\}$ are pairwise distinct points.
\end{thm}
This theorem for $\g=\msl_2$ was proved in the Appendix of \cite{FKLMM}. Their proof
with minor modifications works for general $\g$ as well. For the readers convenience we give the
details below.

We first recall the definition of the Verlinde algebra (which is also referred to as
the fusion ring) and the Verlinde numbers $N_{\mu_1,\dots,\mu_n}^{\la;k}$ (see \cite{TUY}, \cite{V}). Let
$\mu_1,\dots,\mu_n,\la$ be some elements in $P^k_+$. Let $x_1,\dots, x_{n+1}\in\mathbb{CP}^1$
be a set of pairwise distinct points and let $U=\mathbb{CP}^1\setminus\{x_1,\dots,x_{n+1}\}$.
We denote the local coordinates at $x_i$ by $t_i$, $t_i=t-x_i$, where $t$ is a global coordinate
on $\C=\mathbb{CP}^1\setminus\{\infty\}$. If $x_i=\infty$, then we set $t_i=t^{-1}$. Let
$\eO[U]$ be the ring of $\C$-valued functions on $U$ and $\g\T \eO[U]$ be the Lie algebra
of $\g$-valued functions. Then one has the inclusion
\begin{equation}\label{incl}
\g\T\eO[U]\hk \bigoplus_{i=1}^{n+1} \g\T\C[t_i^{-1},t_i]]
\end{equation}
given by the Laurent expansions at points $x_i$ (here $\C[t_i^{-1},t_i]]$
denotes the space of Laurent series in $t_i$). The inclusion \eqref{incl} makes the tensor product
\[
L_{\mu_1}\T \dots \T L_{\mu_n}\T L_{\la^*}
\]
into $\g\T\eO[U]$-module (one thinks of $L_{\mu_i}$ as being attached to the point $x_i$,
$i=1,\dots,n$ and $L_{\la^*}$ being attached to $x_{n+1}$).
We now define the Verlinde numbers
\[
N_{\mu_1,\dots,\mu_n}^{\la;k}=\dim (L_{\mu_1}\T\dots\T L_{\mu_n}\T L_{\la^*}/\g\T\eO[U]).
\]
Here and in what follows for a vector space $V$ and a space $A$ of operators $A\subset \mathrm{End} (V)$
we write $V/A$ for $V/AV$. It is known that the Verlinde numbers are the structure constants
of an associative algebra referred to as the Verlinde algebra. This algebra has a basis $[\la]$
labeled by elements $\la\in P^k_+$ and the multiplication is given by
\[
[\mu_1]\dots [\mu_n]=\sum_{\la\in P_+^k} N_{\mu_1,\dots,\mu_n}^{\la;k} [\la].
\]

Our goal is to prove \eqref{Vn}. We first show that the left hand side of \eqref{Vn}
is finite-dimensional $\g^{\oplus n}=\underbrace{\g\oplus\dots\oplus\g}_n$-module.
\begin{lem}
The quotient space $L_\la/O^c_\bu(\la)$ is finite-dimensional and carries $n$ commuting actions
of $\g$.
\end{lem}
\begin{proof}
Note that the space  $O^c_\bu(\la)$ is stable with respect to the subalgebra $\g\T \C[t^{-1}]$.
Therefore, $L_\la/O^c_\bu(\la)$ is a module of the quotient algebra
\[
\g\T \C[t^{-1}]/\g\T\prod_{j=1}^n (t^{-1}-u_j^{-1})\C[t^{-1}]
\]
(the subalgebra we quotient out acts trivially). This quotient algebra
is isomorphic to $\g^{\oplus n}$: the isomorphism is induced from the composition
\[
\g\T\C[t^{-1}]\to \bigoplus_{j=1}^n \g\T\C[t^{-1}-u_j^{-1}]\to \bigoplus_{j=1}^n \g,
\]
where the last map comes from the Teylor expansion and the first one is the evaluations
at $t=u_j$.

We now prove that $L_\la/O^c_\bu(\la)$ is finite-dimensional. This statement is proved for
$\g=\msl_2$ in \cite{FKLMM}, so we reduce the general situation to the $\msl_2$-case.
Let us reformulate the finite-dimensionality of $L_\la/O^c_\bu(\la)$ in the following way:
Let $L_\la(N)$ be the linear span of the vectors
\[
x_1\T t^{-i_1}\cdots x_s\T t^{-i_s}l_\la
\]
for all $x_1,\dots,x_s\in\g$, $i_1,\dots,i_s>0$, $i_1+\dots +i_s\le N$. Then
$\dim (L_\la/O^c_\bu(\la))<\infty$ if and only if there exists a constant $N_k^\g$ depending on 
a level $k$ such that for 
all $N>N_k^\g$
\begin{equation}\label{x_s}
L_\la(N)\hk L_\la(N-1) + O^c_\bu(\la).
\end{equation}
Let us fix the Chevalley basis $\{e_\al, h_\al, f_\al\}_{\al\in\triangle_+}$
of $\g$.
Let $N_\al$ be the number $N^{\msl_2}_{k_\al}$, where $k_\al$ is the level of the restriction
of a level $k$ $\gh$-module to the subalgebra $\widehat{\msl}_2$, generated by
$e_\al\T t^i$, $h_\al\T t^i$, $f_\al\T t^i$. We prove that if
$N^\g_k=\sum_{\al\in\triangle_+} N_\al$, then \eqref{x_s} holds.

Clearly we can assume that all $x_i$ are from the set $\{e_\al, h_\al, f_\al\}_{\al\in\triangle_+}$.
We use the induction on $s$. If $s=1$ and $i>\sum_{\al\in\triangle_+} N_\al$, then
$x\T t^{-i}\in O^c_\bu(\la)$ by the $\msl_2$ arguments.  Now assume
$i_1+\dots +i_s >\sum_{\al\in\triangle_+} N_\al$. Then there exists a positive root
$\al_0\in\triangle_+$ such that
the sum of powers of $x_i$ with $x_i\in\{e_{\al_0},h_{\al_0},f_{\al_0}\}$
is greater than or equal to $N_{\al_0}$. By induction assumption we can assume that all factors of the form
$e_{\al_0}$, $h_{\al_0}$, $f_{\al_0}$ are on the right of the whole monomial
$x_1\T t^{-i_1}\cdots x_s\T t^{-i_s}$. Now the $\msl_2$ arguments and invariance of $O^c_\bu(\la)$
with respect to $\g\T\C[t^{-1}]$ complete the proof.
\end{proof}

Because of the Lemma above it suffices to prove that
\[
N_{\mu_1,\dots,\mu_n}^{\la;k}=\dim \Hom_{\g^{\oplus n}} (V_{\mu_1}\T\dots\T V_{\mu_n}, L_\la/O^c_\bu(\la)).
\]
The key point here is to consider the parabolic Verma modules . For $\la\in P^k_+$ 
the corresponding parabolic level $k$ Verma
module is a $\gh$-module defined as
\[
W_{\la,k}=\mathrm{Ind}_{\g\T\C[t]\oplus\C K}^{\gh} V_\la,
\]
where $\g\T\C[t]\oplus\C K$ acts on $V_\la$ as follows: $K$ acts by the scalar $k$ and
$\g\T\C[t]$ acts via the evaluation at the $\infty$ map: $x\T 1\mapsto x$, $x\T t^i\mapsto 0$ if $i>0$.

Let $W=\mathbb{CP}^1\setminus\{u_1^{-1},\dots,u_n^{-1},\infty\}$.
\begin{prop}
We have isomorphisms of vector spaces
\begin{align*}
\Hom_{\g^{\oplus n}} (V_{\mu_1}\T\dots\T V_{\mu_n}, L_\la/O^c_\bu(\la)) & \simeq
W_{\mu_1,k}\T\dots \T W_{\mu_n,k}\T L_{\la^*}/\g\T\eO[W],\\
& \simeq L_{\mu_1}\T\dots \T L_{\mu_n}\T L_{\la^*}/\g\T\eO[W].
\end{align*}
\end{prop}
\begin{proof}
The first isomorphism is proved in \cite{FKLMM}, Appendix (the restriction $\g=\msl_2$ is not
used in the proof). The second isomorphism is proved in \cite{Fi} for general $\g$.
\end{proof}

\begin{cor}
Theorem \ref{coinv} is true.
\end{cor}

\section{Applications and further directions}
Combining Theorem \ref{main} and Theorem \ref{coinv} we arrive at the following:
\begin{thm}\label{distinct}
Let $u_1,\dots,u_n$ be pairwise distinct points of $\mathbb{CP}^1\setminus\{0\}$. Then
$A_\bu$ has a natural structure of $\g^{\oplus n}$-module such that
\begin{equation}\label{A}
A_\bu(\la)\simeq
\bigoplus_{\mu_1,\dots,\mu_n\in P_+^k} N_{\mu_1,\dots,\mu_n}^{\la;k} V_{\mu_1}\T \dots \T V_{\mu_n}.
\end{equation}
\end{thm}

\begin{rem}
Note that the existence of $n$ commuting actions of $\g$ on $A_\bu(\la)$ can be
seen from the general formalism of systems of correlation functions, see
\cite{GabGod}, \cite{N}, \cite{GN}.
\end{rem}

We note that if $n=2$ and $\la=0$, then $N_{\mu_1,\mu_2}^{0;k}=\delta_{\mu_1,\mu_2^*}$ and hence
Theorem \ref{distinct} gives
\[
A_\bu(0)\simeq
\bigoplus_{\mu\in P_+^k} V_{\mu}\T  V_{\mu^*},
\]
which is the Zhu theorem for affine Kac-Moody VOAs (see \cite{Z}).
For general $\la$ and still $n=2$ we obtain
\[
A_\bu(\la)\simeq
\bigoplus_{\mu_1,\mu_2\in P_+^k} N_{\mu_1,\mu_2}^{\la;k} V_{\mu_1}\T V_{\mu_2},
\]
which is the Frenkel-Zhu theorem (see \cite{FZ}).

We note that the points $u_j$ in Theorem \ref{distinct} are
assumed to be pairwise distinct. However the spaces $A_\bu(\la)$ are defined for all
tuples of points of $\mathbb{CP}^1\setminus\{0\}$. It is a natural question whether the isomorphism
\eqref{A} holds true if some points coincide. There are several known cases when
it does for all $u_j$.

First, assume $n=2$, $\la=0$ and $\g=\msl_n$ or
$\g=\mathfrak{sp}_{2n}$. Then it is shown in \cite{FFL}, \cite{FL}, \cite{GG}, \cite{F} that
\eqref{A} is true even if $u_1=u_2$.

Now assume $\g=\msl_2$.
\begin{cor}
If $\g=\msl_2$, then \eqref{A} is true for all weights $\la\in P_+^k$ and tuples $\bu$.
\end{cor}
\begin{proof}
It suffices to prove this corollary if all points $u_j$ are equal to $\infty$.
In this case
\begin{equation}\label{infty}
A_\bu(\la)\simeq L_\la/\g\T t^{-n}\C[t^{-1}].
\end{equation}
It is proved in \cite{FKLMM} that the dimension of the right hand side of \eqref{infty}
coincides with the dimension of the right hand side of \eqref{A}. This proves the corollary.
\end{proof}

It would be very interesting to figure out whether \eqref{A} with posiibly coinciding
$u_j$ is true for all classical
$\g$ (note however that it is not true for $E_8$ and $n=2$, $k=1$, $\la=0$, $u_1=u_2$,
see for example \cite{GG}).

\section*{Acknowledgments}
We are grateful to M.Gaberdiel for useful discussion.
This work was partially supported by the Russian President Grant MK-281.2009.1,
by the RFBR Grant 09-01-00058 and by grant Scientific Schools-6501.2010.2, 
by Pierre Deligne fund based on his 2004 Balzan prize in mathematics and by EADS foundation chair in mathematics.


\begin{thebibliography}{99}
\bibitem[BF]{BF}
D.~Ben-Zvi and E.~Frenkel, Vertex algebras and algebraic curves,
Mathematical Surveys and Monographs, {\bf 88}, AMS, 2001.


\bibitem[DM]{DM}
C.~Dong, G.~Mason, {\it Integrability of $C_2$-cofinite vertex operator algebras},
arXiv:math/0601569.

\bibitem[DMS]{DMS}
P.~Di Francesco, P.~Mathieu, D.~S\'en\'echal,
{\it Conformal field theory}, Springer-Verlag, New York, 1997.

\bibitem[F]{F}
E.~Feigin, {\it The PBW filtration}, Represent. Theory 13  (2009), 165-181.

\bibitem[FKLMM]{FKLMM}
B.~Feigin, R.~Kedem, S.~Loktev, T.~Miwa, E.~Mukhin,
{\it Combinatorics of the $\widehat{\msl_2}$ spaces of coinvariants},
Transformation Groups 6 (1), (2001), 25-52.

\bibitem[FFL]{FFL}
B.~Feigin, E.~Feigin, P.~Littelmann,
{\it Zhu's algebras, $C_2$-algebras and abelian radicals},
arXiv:0907.3962.

\bibitem[FL]{FL}
E.~Feigin, P.~Littelmann, {\it Zhu's algebra and the $C_2$-algebra in the symplectic and the
orthogonal cases}, arXiv:0911.2957.

\bibitem[Fi]{Fi}
M.~Finkelberg, {\it An equivalence of fusion categories},
Geom. Funct. Anal. 6:2 (1996), 249-267.

\bibitem[FZ]{FZ}
I.~Frenkel, Y.~Zhu, {\it Vertex operator algebras associated to representations of affine
and Virasoro algebras}, Duke Math. J. 66 (1992), 123--168.

\bibitem[G]{G}
M.~Gaberdiel, {\it An Introduction to Conformal Field Theory},
Rept.Prog.Phys. 63 (2000) 607-667.

\bibitem[GG]{GG}
M.~R.~Gaberdiel, T.~Gannon, {\it Zhu's algebra, the $C_2$ algebra, and twisted modules},
arXiv:0811.3892

\bibitem[GabGod]{GabGod} M.~R.~Gaberdiel, P. Goddard,
{\it Axiomatic conformal field theory}, Commun. Math. Phys. 209
(2000), 549-594.


\bibitem[GN]{GN} M.~R.~Gaberdiel, A.~Neitzke, {\it Rationality, quasirationality
and finite W-algebras}, Commun. Math. Phys. 238 (2003), 305-331.


\bibitem[K1]{K1}
V.~Kac, Infinite dimensional Lie algebras, 3rd ed.,
Cambridge University Press, Cambridge, 1990.

\bibitem[K2]{K2}
V.~Kac, Vertex algebras for begginers, University Lecture Series,
{\bf 10}, 1997.


\bibitem[N]{N} A.~Neitzke,
{\it Zhu's theorem and an algebraic characterization of chiral blocks},
arXiv: hepth/0005144.

\bibitem[TUY]{TUY}
A.~Tsuchiya, K.~Ueno, Y.~Yamada,
{\it Conformal field theory on the universal family of stable curves with gauge symmetry},
Adv. Stud. Pure Math., 19 (1989), 459--466.

\bibitem[V]{V}
E.~Verlinde, {\it Fusion rules and modular transformations in 2D conformal field theory,}
Nuclear Physics B300 (1988), 360--376.

\bibitem[Z]{Z}
Y.~Zhu, {\it Modular invariance of characters of vertex operator algebras},
J. Amer. Math. Soc. 9 (1996), 237--302.
\end{thebibliography}
\end{document}